\documentclass[11pt,a4paper]{article}

\usepackage[utf8]{inputenc}
\usepackage[T1]{fontenc}
\usepackage{amsmath}
\usepackage{amsfonts}
\usepackage{ae}
\usepackage{units}
\usepackage{icomma}
\usepackage{color}
\usepackage{graphicx}
\usepackage{bbm}
\usepackage{amsthm}
\usepackage{amssymb}
\usepackage{cancel}
\usepackage{upgreek}
\usepackage{type1cm}
\usepackage{eso-pic}
\usepackage[breaklinks,pdfpagelabels=false]{hyperref}
\usepackage{xspace}
\usepackage{amsrefs}

\renewcommand{\epsilon}{\varepsilon}

\newtheorem{theorem}{Theorem}[section]

\newtheorem{claim}[theorem]{Claim}
\theoremstyle{definition}

\numberwithin{equation}{section}
\numberwithin{theorem}{section}

\begin{document}

\title{Shotgun edge assembly of random jigsaw puzzles}
\author{Anders Martinsson}
\maketitle

\begin{abstract}
In recent work by Mossel and Ross, it was asked how large $q$ has to be for a random jigsaw puzzle with $q$ different shapes of ``jigs'' to have exactly one solution. The jigs are assumed symmetric in the sense that two jigs of the same type always fit together. They showed that for $q=o(n^{2/3})$ there are a.a.s. multiple solutions, and for $q=\omega(n^2)$ there is a.a.s. exactly one. The latter bound has since been improved to $q\geq n^{1+\varepsilon}$ independently by Nenadov, Pfister and Steger, and by Bordernave, Feige and Mossel. Both groups further remark that for $q=o(n)$ there are a.a.s. duplicate pieces in the puzzle. In this paper, we show that such puzzle a.a.s. has multiple solutions whenever $q\leq \frac{2}{\sqrt{e}}\,n - \omega(\log_2 n)$, even if permuting identical pieces is not considered changing the solution. We further give some remarks about the number of solutions, and the probability of a unique solution in this regime.
\end{abstract}

\section{Introduction}

We construct a jigsaw puzzle in the following manner. Cut an $n\times n$ board into $n^2$ identical squares, which we refer to as \emph{jigsaw pieces}. We color each of the four edges of each piece with one out of $q$ colors with the restriction that whenever two pieces share a side, the corresponding edges should have the same color. Note that this means that we also color the edges along the boundary of the puzzle. In \cite{MR15} it was asked how large $q$ needs to be in order for a uniformly chosen such puzzle to be recovered uniquely from an unordered box of jigsaw pieces.

We will here assume that, when reassembling the puzzle, the pieces can be rotated any multiple of $90^\circ$, but not flipped upside down. In other words, given a single jigsaw piece there is no way to tell which of the four orientations it originally had in the puzzle. We remark that this assumption differs between previous papers on the topic. In \cite{MR15} it seems to be tacitly assumed that rotations are allowed, in \cite{NPS16} it is explicitly stated that they allow rotations, and in \cite{BFM16} it is stated that they, for sake of simplicity, assume rotations are not allowed.

The notion of ``unique recovery'' needs some elaboration. Note that rotating the original configuration of the puzzle a multiple of $90^\circ$ will result in the same pieces. Hence the best we can hope for is to recover the original configuration up to rotation of the entire puzzle. There are arguably two natural interpretations of a unique recovery. The first interpretation is that the position and orientation of each piece can be uniquely determined up to rotation of the entire puzzle. Note that this is trivially impossible if the puzzle contains two identical pieces (or one piece where opposite edges have the same color and $n\geq 2$). The second interpretation is that the ``image'' of the puzzle can be uniquely determined, that is, all solutions yield the same edge coloring up to rotation of the entire puzzle, but may differ by a reshuffling of identical pieces. Following the terminology in \cite{BFM16} we will refer to the two interpretations as a \emph{unique vertex assembly} and a \emph{unique edge assembly} respectively. Note that a puzzle has a unique edge assembly if it has a unique vertex assembly.

For functions $f:\mathbb{N}\rightarrow\mathbb{R}$ and $g:\mathbb{N}\rightarrow\mathbb{R}$, $f_n = O(g_n)$ denotes that $f_n/g_n$ is bounded for sufficiently large $n$, and $f_n = o(g_n)$ that $f_n/g_n\rightarrow 0$ as $n\rightarrow\infty$. Furthermore, assuming $f, g\geq 0$, $f_n=\Omega(g_n)$ denotes that $f_n/g_n$ is bounded away from $0$ for sufficiently large $n$, and $f_n=\omega(g_n)$ that $f_n/g_n\rightarrow\infty$ as $n\rightarrow\infty$. We use $\mathbb{P}(\cdot), \mathbb{E}[\cdot],$ and $\mathbb{H}(\cdot)$ to denote probability, expectation and entropy respectively. We say that a sequence of events, $\{E_n\}_{n=1}^\infty$, occurs asymptotically almost surely (a.a.s.) if $\mathbb{P}(E_n)\rightarrow 1$ as $n\rightarrow\infty$.

It was shown in \cite{MR15} that the puzzle a.a.s. has neither a unique vertex nor edge assembly when $2\leq q=o(n^{2/3})$, and a.a.s. has both if $q=\omega(n^2)$. Recently two papers \cites{NPS16, BFM16} considering this problem were published on the arxiv, both on May 11:th 2016, and both proving essentially the same result: for $q=n^{1+\varepsilon}$, the puzzle a.a.s. has both a unique vertex and edge assembly for any fixed $\varepsilon>0$, and further noted that for $q=o(n)$ there are a.a.s. duplicate jigsaw pieces, and hence the puzzle does not have a unique vertex assembly. We note however that the result by \cite{NPS16} is slightly stronger as they allow for the pieces to be rotated.

The aim of this paper is to prove the following result.
\begin{theorem}
For $2\leq q\leq\frac{2}{\sqrt{e}}\,n-\omega(\log_2 n)$ the puzzle a.a.s. has neither a unique vertex nor edge assembly.
\end{theorem}

This will be shown in the next section. Our proof is based comparing the entropy of the assembled puzzle to the entropy of the unordered set of pieces together with a concentration result. In the final section, we give some concluding remarks on the result including what happens if the pieces are not allowed to be rotated.

\section{Proof of Theorem 1.1}

Let $\mathcal{J}$ denote the set of possible types of jigsaw pieces. For $J\in\mathcal{J}$, we define $r(J)$ to be the number of distinct edge colorings that can be obtained by rotating the piece, that is $r(J)=1$ if all sides of $J$ have the same color, $r(J)=2$ if opposite but not adjacent sides have the same color, and $r(J)=4$ otherwise. There are $q$ types of jigsaw pieces such that $r(J)=1$, $q(q-1)/2$ such that $r(J)=2$, and $(q^4-q^2)/4$ such that $r(J)=4$.

Let us consider the problem from an information theoretic point of view. We are given a box of unordered jigsaw pieces. Effectively, we are told how many pieces there are of each type in the puzzle, that is, we are given the collection of random variables $\{X_J\}_{J\in\mathcal{J}}$ where $X_J$ denotes the number of jigsaw pieces of type $J$ in the puzzle. The problem is to determine the image of the original solution up to rotation, that is, to determine the equivalence class of edge colorings as above that contains the original edge coloring, where two edge colorings are equivalent if they can be rotated to each other. We will refer below to these random variables as the \textsf{BOX} and the \textsf{IMG} respectively. Let us for simplicity write $q=n^{\beta}$.

\begin{claim}\label{claim:HBOX}
For $\beta \geq \frac{1}{2}+\frac{\log_2\log_2 n}{\log_2 n}$, equivalently $q\geq \sqrt{n}\log_2 n$, we have
$$\mathbb{H}(\mathsf{BOX}) \leq (4\beta-2)n^2\log_2 n - (2-\log_2 e)n^2 + O( n^{4-4\beta}\log_2 n+n^{2-2\beta}).$$
\end{claim}
\begin{proof}
Since entropy is subadditive, we can write
$$\mathbb{H}(\mathsf{BOX}) \leq \sum_{J\in\mathcal{J}} \mathbb{H}(X_J).$$

By linearity of expectation, we have
$$\mathbb{P}(X_J=1) \leq \mathbb{E}X_J = \frac{r(J)n^2}{ q^4}=r(J)n^{2-4\beta},$$
and
$$\mathbb{P}(X_J=0) \geq 1-\mathbb{E} X_J = 1-r(J)n^{2-4\beta}.$$
Slightly more elaborately, we can bound $\mathbb{P}(X_J\geq 2)$ by estimating the expected number of pairs of distinct jigsaw pieces of type $J$ in the original solution. We distinguish between pairs of pieces that were not, respectively were, adjacent in the original solution. In the former case, there are $O(n^4)$ pairs, and the probability that both have type $J$ is $O(q^{-8})$. In the latter case, there are $O(n^2)$ pairs, with corresponding probability $O(q^{-7})$. Hence
$$\mathbb{P}(X_J\geq 2) = O(n^{4-8\beta}+n^{2-7\beta}).$$
Note that the lower bound on $q$ implies that $r(J)n^{2-4\beta}$ and $O(n^{4-8\beta}+n^{2-7\beta})$ are $o(1)$.

Now, we have
\begin{equation*}
\begin{split}
\mathbb{H}(X_J) &= -\mathbb{P}(X_J=0) \log_2 \mathbb{P}(X_J=0) -\mathbb{P}(X_J=1) \log_2 \mathbb{P}(X_J=1)\\
&\qquad - \sum_{k=2}^{n^2} \mathbb{P}(X_J=k) \log_2 \mathbb{P}(X_J=k)\\
&\leq -\mathbb{P}(X_J=0) \log_2 \mathbb{P}(X_J=0) -\mathbb{P}(X_J=1) \log_2 \mathbb{P}(X_J=1)\\
&\qquad - \mathbb{P}(X_J\geq 2) \log_2 \frac{\mathbb{P}(X_J\geq 2)}{n^2-1},
\end{split}
\end{equation*}
where the last step follows by Jensen's inequality. As $-p \log_2 p$ is increasing for $0\leq p\leq e^{-1}$ and decreasing for $e^{-1}\leq p \leq 1$, we have, assuming $n$ is sufficiently large
$$-\mathbb{P}(X_J=0) \log_2 \mathbb{P}(X_J=0) \leq (\log_2 e) r(J) n^{2-4\beta} + O(n^{4-8\beta}),$$
$$-\mathbb{P}(X_J=1) \log_2 \mathbb{P}(X_J=1) \leq r(J)(4\beta-2)n^{2-4\beta}\log_2 n-r(J)(\log_2 r(J)) n^{2-4\beta},$$
and
$$ - \mathbb{P}(X_J\geq 2) \log_2 \frac{\mathbb{P}(X_J\geq 2)}{n^2-1} = O( \beta(n^{4-8\beta}+n^{2-7\beta}) \log_2 n).$$
The claim follows by summing over all $J\in\mathcal{J}$.
\end{proof}

\begin{claim}\label{claim:HIMG}
$$\mathbb{H}(\mathsf{IMG}) = 2\beta n(n+1) \log_2 n - 2 + n^{-\beta n(n+1)} + n^{-3\beta n(n+1)/2}.$$
\end{claim}
\begin{proof} The probability that the puzzle is originally colored in accordance with a certain equivalence class of edge colorings is $4 q^{-2n(n+1)}$ if all rotations of an edge coloring in the class are distinct, $2 q^{-2n(n+1)}$ if it is symmetric under a $180^\circ$ but not $90^\circ$ rotation, and $ q^{-2n(n+1)}$ if symmetric under a $90^\circ$ rotation. These cases occur with respective probabilities $1-q^{-n(n+1)}$, $q^{-n(n+1)}-q^{-3n(n+1)/2}$ and $q^{3n(n+1)/2}$ respectively. This is because symmetry under a $90^\circ$ rotation is equivalent to that groups of four edges have the same color, and symmetry under $180^\circ$ groups of two. It follows that
\begin{equation*}
\begin{split}
\mathbb{H}(\textsf{IMG}) = &-(1-q^{-n(n+1)})\log_2 (4q^{-2n(n+1)})\\
&-(q^{-n(n+1)}-q^{-3n(n+1)/2})\log_2 (2 q^{-2n(n+1)})\\
&-q^{-3n(n+1)/2}\log_2 q^{-2n(n+1)},
\end{split}
\end{equation*}
which simplifies to the expression above.
\end{proof}

The preceding claims shows that for $\beta\leq 1-\varepsilon$ there is less information in the box than in the image. At least heuristically this should mean that a.a.s. there is no unique edge assembly. To conclude the proof of our main result, the idea is that if, for given $q$, the probability of a unique edge assembly is not too small, then very little additional information is needed to reconstruct the original solution, meaning that this can only happen for $\beta$ close to $1$.

Suppose that we, besides the box of unordered jigsaw pieces, are given the following information: a list of pairs of types of jigsaw pieces, and a list of instructions how to repaint certain edges in an assembled puzzle. The idea is that the first list tells us how to exchange some of the pieces in the box to obtain a puzzle with a unique edge assembly ``close to'' the original solution of the first puzzle, and the second list tells us how to obtain the original solution exactly from this. This is clearly always possible to do as we could, for instance, first replace the entire puzzle with one where all edges have the same color, and then describe how to repaint all edges. Below, we will let $\textsf{INS}$ denote any choice of lists of instructions as above of minimal total length.

Let $Y$ denote the total length of \textsf{INS}. Note that changing the color of one edge in the original edge coloring changes $Y$ by at most $3$ -- we can always modify a set of instructions for one of the edge colorings to work for the other by replacing two extra jigsaw pieces in the beginning and repainting one extra edge at the end. Hence by Azuma's inequality, see Theorem 7.4.2 in \cite{AS08}, we have
\begin{equation*}
\mathbb{P}( Y \leq \mathbb{E} Y-t ) \leq \exp\left(-\frac{t^2}{36 n(n+1)}\right),
\end{equation*}
for any $t\geq 0$. Observing that the puzzle has a unique edge assembly if and only if $Y=0$, it follows by putting $t=\mathbb{E}Y$ that
\begin{align*}
\mathbb{P}(\text{unique edge assembly}) \leq \exp\left( -\frac{ (\mathbb{E}Y)^2}{36n(n+1)} \right).
\end{align*}

Let us bound $\mathbb{H}(\textsf{INS})$. The lists are of lengths at most $n^2$ and $2n(n+1)$ respectively, hence the entropy of list lengths is $O(\log_2 n)$. Conditioned on these lengths, the entropy of an entry in either list can be bounded by $\max(8\beta, 3+\beta) \log_2 n$ -- an instruction is either $8$ colors, or the position of one edge and a color. Hence, for $\beta \geq \frac{1}{2}+\frac{\log_2\log_2 n}{\log_2 n}$
we have
\begin{align*}
&O(\log_2 n) + \max(8\beta, 3+\beta) \mathbb{E} Y \log_2 n\\
&\qquad \geq \mathbb{H}(\mathsf{INS}) \geq \mathbb{H}(\mathsf{IMG})-\mathbb{H}(\mathsf{BOX})\\
&\qquad \geq (2-2\beta)n^2 \log_2 n + (2-\log_2 e) n^2 + 2\beta n\log_2 n + O(n^{4-4\beta}\log_2 n + 1),
\end{align*}
where the error term $O(n^{2-2\beta})$ from Claim \ref{claim:HBOX} disappears as it is always smaller than one of $n^{4-4\beta}\log_2 n$ and $1$. This implies that
$$\max(8\beta, 3+\beta) \mathbb{E}Y \geq (2 - 2\beta)n^2 + \frac{(2-\log_2 e) n^2}{\log_2 n} + 2\beta n + O( n^{4-4\beta}+1 ).$$
It follows that if $\frac{1}{2}+\frac{\log_2\log_2 n}{\log_2 n} \leq \beta \leq 1+\frac{2-\log_2 e}{2\log_2 n}-\omega\left(\frac{1}{n}\right)$, equivalently $\sqrt{n}\log_2 n \leq q \leq \frac{2}{\sqrt{e}}\,n-\omega(\log_2 n)$, we have $\mathbb{E} Y = \omega(n)$. Hence, the probability of a unique edge assembly tends to zero. The remaining case of $2\leq q \leq \sqrt{n}\log_2 n$ is already covered in \cite{MR15}.\qed

\section{Final remarks}
Our proof can easily be modified to the case of pieces with fixed orientation. In this case, it follows by somewhat simpler entropy calculations that, for any $q\geq \sqrt{n}\log_2 n$,
$$\mathbb{H}(\textsf{BOX}) \leq (4\beta-2)n^2\log_2 n + (\log_2 e) n^2 + O(\beta( n^{4-4\beta}+n^{2-3\beta})\log_2 n),$$
and for any $q$,
$$\mathbb{H}(\textsf{IMG}) = 2\beta n(n+1)\log_2 n.$$
As a consequence, there is a.a.s. neither a unique vertex nor edge assembly for $2\leq q \leq \frac{1}{\sqrt{e}} n - \omega(\log_2 n)$.

The bounds $q\leq\frac{2}{\sqrt{e}}\,n-\omega(\log_2 n)$ and $q \leq \frac{1}{\sqrt{e}}\,n - \omega(\log_2 n)$ are optimal in the sense that they cannot be improved by considering the error terms for $\sum_{J\in\mathcal{J}} \mathbb{H}(X_J)$ and $\mathbb{H}(\textsf{IMG})$ more closely. There might however be other avenues for improvement: How large is the difference between $\mathbb{H}(\textsf{BOX})$ and $\sum_{J\in\mathcal{J}} \mathbb{H}(X_J)$?. Can we improve the size or concentration of the instructions?

It is worth noting that, as an implication of our main result, in the window $\Omega(n)=q\leq\frac{2}{\sqrt{e}}\,n-\omega(\log_2 n)$, the puzzle can a.a.s. be assembled in multiple ways, while with probability bounded away from zero the puzzle contains no duplicate pieces.

The random jigsaw puzzle has a phase transition at $\beta=\frac{1}{2}$ where, above this value most jigsaw pieces in the puzzle are unique, and below it $X_J=\omega(1)$ for most $J\in\mathcal{J}$. This is reflected in the entropy of \textsf{BOX}. It follows from a simple coupling argument that the entropy of \textsf{BOX} with $q$ colors is at most the entropy with $2q$ colors. Hence by Claim \ref{claim:HBOX}, we have for $q\leq \sqrt{n}\log_2 n$ that $\mathbb{H}(\textsf{BOX})=o(n^2\log_2 n)$. For sufficiently small $q$ this can be improved further by bounding $\mathbb{H}(X_J)$ by $O(\log_2 n)$, which yields $\mathbb{H}(\textsf{BOX})=O(q^4 \log_2 n)$. Note that these estimates mean that, for $2\leq q\leq \sqrt{n}$, we have $\mathbb{H}(\textsf{BOX}) = o(\mathbb{H}(\textsf{IMG}))$. Morally, the box contains essentially no information about the original image.

Using the entropy estimates in this regime together with Claims \ref{claim:HBOX} and \ref{claim:HIMG}, it follows that for any $2\leq q \leq n$, we have $$\mathbb{H}(\textsf{IMG})-\mathbb{H}(\textsf{BOX}) \geq  (2-o(1))\min\left( \beta, 1-\beta \right)n^2 \log_2 n + \omega(n\log_2 n).$$
Thus, for any $2\leq q \leq n$, we have $\mathbb{E}Y= \Omega( \min(\beta, 1-\beta) n^2)$, and hence the explicit estimate
$$\mathbb{P}(\text{unique edge assembly}) \leq \exp\left(-\Omega( \min(\beta, 1-\beta)^2 n^2\right).$$

Another interesting aspect is how many distinguishable solutions the puzzle typically has. We can approach this by modifying the instructions so that, instead of replacing pieces to obtain a puzzle with a unique edge assembly, we do so to obtain one with at most $N$ distinguishable solutions, and append instructions that tell us which solution we should consider. As before, this implies that
$$O(\log_2 n) + \max(8\beta, 3+\beta) \mathbb{E}Y \log_2 n + \log_2 N \geq \mathbb{H}(\textsf{IMG})-\mathbb{H}(\textsf{BOX}).$$
Hence, taking $\log_2 N = (2-o(1))\min\left( \beta, 1-\beta \right)n^2 \log_2 n$, it follows that for any $2\leq q \leq n$, the puzzle has a.a.s. more than $$N=\left(\min(q^2,\frac{n^2}{q^2})\right)^{(1-o(1))n^2}$$ distinguishable solutions.

To get a feel for what this result means, suppose that we try to solve the puzzle greedily by just randomly placing pieces from the box at $(1, 1), (2, 1)$, $\dots, (n, 1), (1, 2), (2, 2)$ and so on with the only restriction that the new piece should fit together with any adjacent piece already in place. For $q=o(\sqrt{n})$, essentially all types of pieces exist in the box, so we can pick any of the $\Theta(q^2)$ types of pieces that would fit together with the two adjacent pieces already in place. For $q=\omega(\sqrt{n})$, there are $\Theta(n^2)$ types of pieces in the box, and we should expect that a proportion of $\Theta(q^{-2})$ of these will fit at the current position. Assuming this method never gets stuck, we would get about $\left(\min(q^2,\frac{n^2}{q^2})\right)^{n^2}$ distinct solutions.

\section*{Acknowledgements} 
I want to thank Johan Tykesson and Peter Hegarty for valuable comments.

\end{document}